\documentclass[a4paper,11pt,leqno]{article}
\usepackage{amssymb}
\usepackage{graphicx}
\usepackage{amsmath}
\usepackage{pstricks}
\usepackage{psfrag}
\usepackage[latin5]{inputenc}
\usepackage{setspace}

\newcounter{theorem}
\newtheorem{theorem}{Theorem}[section]

\newtheorem{corollary}{Corollary}[section]

\newtheorem{definition}{Definition}[section]
\newtheorem{example}{Example}[section]

\newtheorem{proposition}{Proposition}[section]

\newenvironment{proof}[1][Proof]{\textbf{#1.} }{\rule{0.5em}{0.5em}}

\begin{document}

\begin{center}

{\Large \bf On the Directional Derivative Sets  and \\ Differentials of the Set Valued Maps}

\vspace{5mm}

Serpil Altay, Nihal Ege, Anar Huseyin, Nesir Huseyin
\end{center}

\vspace{2mm}
{\small

\vspace{4mm}

{\bf Abstract.} In this paper directional derivative sets and differentials of a given set valued map are studied. Relations between the set valued map and compact subsets of the directional derivative sets of the given map are investigated. Upper and lower contingent cones of some plane sets are calculated.

{\bf Mathematics Subject Classification 2010:} 26E25, 54C60

{\bf Key words:} Set valued map, directional derivative set, differential. }
\vspace{2mm}

\section{Introduction}

Set valued maps arise in the mathematical models of physics, economics and biology. They are important tools for investigation of the optimal control, optimization and game theory problems. In studying of these problems  it is often required to deal with differential or directional derivative sets of the given set valued map. In general, the differential notions of the set valued maps are based on the various types of tangent and contingent cones. In this paper, the upper and lower differentials of the set valued map are defined via upper and lower Bouligand contingent cones, which are used in many problems of set valued and nonsmooth analysis.  The definition of directional derivative set is different from the definition of differential, but they also are closely connected  with  the concept of contingent cones. In the presented paper, the  connections between differentials (upper and lower)  and directional derivative sets (upper and lower) of a set valued map are studied. Upper and lower contingent cones of the sets given on the plane  are calculated. The paper is organized as follows:

In Section 2 the upper and lower contingent cones of the sets given on the plane are calculated (Example \ref{ex1} and Example \ref{ex2}). In Section 3 connections between directional upper (lower) derivative sets  and upper (lower) differentials are given. It is shown that if the set valued map is not locally Lipschitz continuous, then lower derivative set in the direction $p$ and the value of the lower differential at $p$ does not coincide (Example \ref{ex3}). For scalar variable set valued maps, it is proved that upper derivative set in the direction $p$ and the value of the upper differential at $p$ are equal (Theorem \ref{teo23}). In Section 4 the properties of the compact subsets of the directional derivative sets and differentials are investigated. The Hausdorff deviation of the cone generated by a compact subset of the lower directional derivative set from the given set valued map is estimated  (Theorem \ref{teo31}).

\setcounter{equation}{0}

\section{Upper and Lower Contingent Cones}

Let us give the definitions of upper and lover contingent cones.
\begin{definition} 
Let $X$ be a Banach space,  $K\subset X$ be a closed set  and  $x\in X$. The sets
\begin{eqnarray*} T^{U}_{K}(x)=\left\{u \in X:\liminf_{\delta \rightarrow 0^{+}}\frac{1}{\delta} d(x+\delta u, K)=0\right\}
\end{eqnarray*} and
\begin{eqnarray*}
T^{L}_{K}(x)=\left\{u \in X:\lim_{\delta \rightarrow 0^{+}}\frac{1}{\delta} d(x+\delta u, K)=0\right\}
\end{eqnarray*}
are called the upper and lower contingent cone of the set $K$ at $x\in X$ respectively,  where
$\displaystyle d(y,K)=\inf_{z\in K}\left\| y-z \right\|$, i.e. $d(y,K)$ is the distance from the point $y$ to the set $K$.
  \end{definition}

$T^{U}_{K}(x)$ and $T^{L}_{K}(x)$ are closed cones in the space $X.$
It is obvious that $u\in T^{U}_{K}(x)$ if and only if there exist sequences $\left\{\delta_{i}\right\}_{i=1}^{\infty}$ and $\left\{s_{i}\right\}_{i=1}^{\infty}$ such that $\delta_{i}\rightarrow 0^+$  and  $s_{i}\rightarrow 0$ as $i\rightarrow +\infty$ and the inclusion $x_{i}=x+\delta_{i}u+\delta_{i}s_{i} \in K$ is satisfied for every $i=1,2,\ldots$.

Similarly,  $u\in T^{L}_{K}(x)$ if and only if there exist $\delta_{*}>0$ and $s(\cdot):(0,\delta_{*}]\rightarrow X$ such that $s(\delta)\rightarrow 0$ as $\delta \rightarrow 0^+$ and the inclusion $x(\delta)=x+\delta u+\delta s(\delta) \in K$ is verified for every $\delta \in [0,\delta_{*}]$.

Now, let us  calculate contingent cones of the sets given on the plane.
\begin{example} \label{ex1}
Let the set $K \subset \mathbb{R}^{2}$ be given as
\begin{eqnarray} \label{ka}
K=\left\{ \left(\frac{1}{n},\frac{1}{n}\right) \in \mathbb{R}^{2}:n\in \mathbb{N}\right\}\bigcup \left\{(0,0)\right\},
\end{eqnarray} where $\mathbb{N}=\left\{1,2,\ldots \right\}.$

According to (AUBIN, J.P.; FRANKOWSKA, H. $-$\emph{Set Valued Analysis},
Birkhauser, Boston, 2009. p.161, Fig. 4.4),  $T^{L}_{K}(0,0)=\left\{(0,0)\right\}$ and $T^{U}_{K}(0,0) =\big\{(\alpha , \alpha) \in  \mathbb{R}^{2}:\alpha \geq 0 \big\}$. Note that

\begin{eqnarray} \label{equ}T^{L}_{K}(0,0)=T^{U}_{K}(0,0)=\left\{(\alpha , \alpha)\in \mathbb{R}^{2}:\alpha \geq 0 \right\},
\end{eqnarray} and hence the equality $T^{L}_{K}(0,0)=\left\{(0,0)\right\}$  is not true. The validity of equality (\ref{equ}) can be shown in the following way.

  First of all, we show that $(1,1) \in T^{L}_{K}(0,0).$
  Choose an arbitrary sequence $\{\delta_{i}\}_{i=1}^{\infty}$ such that  $\delta_i \rightarrow 0^{+}$ as $i\rightarrow +\infty$. Then for each $\delta_i$ there exists $m_{i}\in \mathbb{N}$ such that
  \begin{equation}\label{1}
     \delta_{i}\in \left(\frac{1}{m_{i}+1},\frac{1}{m_{i}}\right].
  \end{equation}

   Since $\delta_i \rightarrow 0^{+}$ as $i\rightarrow +\infty$ then $m_i \rightarrow +\infty$ as $i\rightarrow +\infty.$  Thus, (\ref{1}) implies
   \begin{eqnarray*}
   \lim_{i \rightarrow \infty}\frac{1}{\delta_{i}}d\left( \left(0,0\right)+\delta_{i}(1,1),K\right)  &=&
     \lim_{i \rightarrow \infty}\frac{1}{\delta_{i}}d\left(\left(\delta_{i},\delta_{i}\right),K
     \right) \\ & \leq & \lim_{i \rightarrow \infty}\frac{1}{\delta_{i}}d\left( \left( \frac{1}{m_{i}},\frac{1}{m_{i}}\right), \left(\frac{1}{m_{i}+1},\frac{1}{m_{i}+1}\right) \right)\\
     & = & \lim_{i \rightarrow \infty}\frac{1}{\delta_{i}}\sqrt{2}\left(\frac{1}{m_{i}}-\frac{1}{m_{i}+1} \right)\\
     &=& \lim_{i \rightarrow \infty}\frac{1}{\delta_{i}}\frac{\sqrt{2}}{m_{i}(m_{i}+1)} \\
     &\leq& \lim_{i \rightarrow \infty} \left(m_{i}+1 \right)\sqrt{2}\frac{1}{m_{i}(m_{i}+1)}\\
     &=&\lim_{i \rightarrow \infty}\sqrt{2}\frac{1}{m_{i}}=0 .
   \end{eqnarray*}

   Since the sequence $\{\delta_{i}\}_{i=1}^{\infty}$ is arbitrarily chosen, we conclude that
  \begin{eqnarray*} \lim_{\delta\rightarrow 0^{+}}\frac{1}{\delta}d\left((0,0)+\delta (1,1),K\right)=0,
  \end{eqnarray*}
which proves that $(1,1) \in T^{L}_{K}(0,0)$. Since $(1,1) \in T^{L}_{K}(0,0)$ and  $T^{L}_{K}(0,0)\subset \mathbb{R}^2$ is a cone, then we obtain that
$(\alpha,\alpha)\in T^{L}_{K}(0,0)$ for every $\alpha \geq 0$. Thus the inclusion
\begin{eqnarray}\label{eq14*}
\left\{(\alpha,\alpha)\in \mathbb{R}^{2}:\alpha \geq 0\right\}\subset T^{L}_{K}(0,0)
\end{eqnarray}
is verified.

Now, we choose an arbitrary  $(\alpha ,\beta) \in T^{U}_{K}(0,0)$. Then there exist sequences $\left\{\delta_{i}\right\}_{i=1}^{\infty}$ and $\left\{(p_{i},q_{i})\right\}_{i=1}^{\infty}$ such that $\delta_i \rightarrow 0^{+}$, $(p_{i},q_{i})\rightarrow (0,0)$ as $i\rightarrow +\infty$ and \begin{eqnarray}\label{eq14} (x_{i},y_{i})=(0,0)+\delta_{i}(\alpha,\beta)+\delta_{i}(p_{i},q_{i})\in K
\end{eqnarray} for every $i=1,2,\ldots$ According to (\ref{ka})  we have that $x_i=y_i$ for every $i=1,2,\ldots$. It follows from (\ref{eq14})
that
\begin{equation}\label{eq15}
    \alpha+p_{i}=\beta+q_{i}
 \end{equation}
for every $i=1,2,\ldots$ Since  $p_{i}\rightarrow 0,$  $q_{i}\rightarrow 0$ as $i\rightarrow +\infty$, then we obtain from (\ref{eq15}) that  $\alpha=\beta$, and hence again by virtue of (\ref{eq15}) $p_i=q_i$ for every $i=1,2,\ldots$. Concluding, we obtain from (\ref{eq14}) that
\begin{eqnarray}\label{eq16} (\alpha,\alpha)+(p_{i},p_{i})\in \frac{1}{\delta_i}K
\end{eqnarray} for every $i=1,2,\ldots$ Since $p_{i}\rightarrow 0$ as $i\rightarrow +\infty$, then  (\ref{ka}) and (\ref{eq16}) yield that $\alpha \geq 0.$ Thus, we have
\begin{eqnarray}\label{eq17}T^{U}_{K}(0,0) \subset \left\{(\alpha,\alpha)\in \mathbb{R}^{2}:\alpha \geq 0\right\}.
\end{eqnarray}

Since $T^{L}_{K}(0,0) \subset T^{U}_{K}(0,0)$, then (\ref{eq14*}) and (\ref{eq17}) implies the validity of equality (\ref{equ}).

\end{example}

We now give an  example which illustrates that lover and upper contingent cones not always coincide.
\begin{example} \label{ex2}
Let the set $\Omega \subset \mathbb{R}^2$ be defined as
\begin{eqnarray}\label{ome}
\Omega=\left\{ \left(\frac{1}{(2n)!},\frac{1}{(2n)!}\right):n\in \mathbb{N}\right\}\bigcup \left\{(0,0)\right\}.
\end{eqnarray}

Let us show that $T^{U}_{\Omega}(0,0) \not \subset T^{L}_{\Omega}(0,0)$.

At first, we will show that $(1,1)\notin T^{L}_{\Omega}(0,0)$. Let us choose a  sequence $\{\delta_{k}\}_{k=1}^{\infty}$, where  $\displaystyle \delta_{k}=\frac{1}{(2k+1)!}$. By virtue of (\ref{ome}) we have
\begin{eqnarray*}
 && \lim_{k \rightarrow \infty}\frac{1}{\delta_{k}}d \big((0,0)+\delta_{k}(1,1),\Omega \big)=\lim_{k \rightarrow \infty}(2k+1)! \cdot d\left( \left(\frac{1}{(2k+1)!} \ ,\frac{1}{(2k+1)!}\right),\Omega \right)\\
  &=&\lim_{k \rightarrow \infty}(2k+1)! \cdot \left\| \left(\frac{1}{(2k+1)!} \ ,\frac{1}{(2k+1)!}\right)-\left(\frac{1}{(2k+2)!} \ ,\frac{1}{(2k+2)!}\right)\right\|\\
  &=&\lim_{k \rightarrow \infty}(2k+1)! \cdot \sqrt{2}\left(\frac{1}{(2k+1)!}-\frac{1}{(2k+2)!}\right)\\
  &=&\sqrt{2} \lim_{k \rightarrow \infty}\left(1-\frac{1}{2k+2}\right)\\
  &=&\sqrt{2}> 0,
\end{eqnarray*}
and hence $(1,1) \notin T^{L}_{\Omega}(0,0)$.

Let  $\displaystyle \delta_k= \frac{1}{(2k)!},$ $k=1,2,\ldots$ Then
\begin{eqnarray*}\liminf_{\delta\rightarrow 0^+}\frac{1}{\delta}d \big((0,0)+\delta(1,1),\Omega \big) &\leq& \lim_{k \rightarrow \infty}\frac{1}{\delta_{k}}d \big( (\delta_{k},\delta_{k}),\Omega\big)\\ &=&\lim_{k \rightarrow \infty}(2k)! \cdot d\left( \left(\frac{1}{(2k)!},\frac{1}{(2k)!} \right),\Omega \right)=0,
\end{eqnarray*}
and hence  $(1,1) \in T^{U}_{\Omega}(0,0)$. Since $(1,1) \not \in T^{L}_{\Omega}(0,0)$, we have that  $T^{L}_{\Omega}(0,0) \not \subset T^{U}_{\Omega}(0,0).$

Note that similarly to the Example \ref{ex1} it is possible to show that

\begin{equation*}
T^{L}_{\Omega}(0,0)=\left\{(0,0)\right\}, \ \ \ T^{U}_{\Omega}(0,0)=\left\{(\alpha,\alpha)\in \mathbb{R}^{2}:\alpha \geq 0\right\}.
\end{equation*}
\end{example}

\setcounter{equation}{0}

\section{Upper and Lower Directional Derivative Sets}

In this section upper  and lower differentials and  directional derivative sets of the set valued maps are investigated. The graph of the set valued map $F(\cdot):X \rightsquigarrow Y$ is denoted by $gr  F(\cdot)$ and is defined as $gr  F(\cdot)=\left\{(x,y)\in X\times Y:y\in F(x)\right\},$ where $X$ and $Y$ are Banach spaces.

\begin{definition} \label{loclip}
  Let $F(\cdot):X \rightsquigarrow Y$ be a given set valued map. Assume  that for every $x \in X$ there exist $L_x \geq 0$ and $r_x>0$ such that for every $y \in B(x,r_x)$ and $z \in B(x,r_x)$ the inequality
  \begin{eqnarray*}
  h\left(F(y), F(z)\right) \leq L_x \cdot d(y,z)
  \end{eqnarray*}
  is satisfied. Then the set valued map $F(\cdot)$ is called locally Lipschitz continuous.

  Here $h\left(F(y), F(z)\right)$ denotes the Hausdorff distance between the sets  $F(y)$ and $F(z)$, $B(x,r_x) =\left\{y\in X: d(y,x)<r_x\right\}.$
\end{definition}

\begin{definition} \label{updif}
  Let $F(\cdot):X \rightsquigarrow Y$ be a set valued map, $(x,y)\in X\times Y$. The set valued map $D^{U}F(x,y)|(\cdot):X \rightsquigarrow Y$ satisfying the equality
  \begin{eqnarray*}
  gr D^{U}F(x,y)|(\cdot)=T^{U}_{gr  F(\cdot)}(x,y)
  \end{eqnarray*}
  is called the upper differential of the set valued map $F(\cdot)$ at the point $(x,y).$

  Here $T^{U}_{gr  F(\cdot)}(x,y)$ is upper contingent cone of the set $gr  F(\cdot)$ at the point $(x,y).$
\end{definition}

\begin{definition} \label{lodif}
  Let $F(\cdot):X \rightsquigarrow Y$ be a set valued map, $(x,y)\in X\times Y$. The set valued map $D^{L}F(x,y)|(\cdot):X \rightsquigarrow Y$ satisfying the equality
  \begin{eqnarray*}
  gr D^{L}F(x,y)|(\cdot)=T^{L}_{gr F(\cdot)}(x,y)
  \end{eqnarray*}
  is called the lower differential of the set valued map $F(\cdot)$ at the point $(x,y).$

   Here $T^{L}_{gr  F(\cdot)}(x,y)$ is lower contingent cone of the set $gr  F(\cdot)$ at the point $(x,y).$
\end{definition}

Now, let us formulate definitions of the upper and lower directional derivative sets of a given set valued map.

\begin{definition} \label{upder}
  Let $F(\cdot):X \rightsquigarrow Y$ be a set valued map, $(x,y)\in X\times Y$ and $p\in X$. The set $\displaystyle \frac{\partial^{U}F(x,y)}{\partial p}$ defined by
  \begin{eqnarray*}\frac{\partial^{U}F(x,y)}{\partial p}=\left\{u\in Y:\liminf_{\delta\rightarrow 0^{+}}\frac{1}{\delta}d \big(y+\delta u,F(x+\delta p )\big)=0\right\}
  \end{eqnarray*} is called
  upper derivative set of the set valued map $F(\cdot)$ at the point $(x,y)$ in the direction $p.$
  \end{definition}

\begin{definition} \label{loder}
  Let $F(\cdot):X \rightsquigarrow Y$ be a set valued map, $(x,y)\in X\times Y$ and $p\in X$. The set $\displaystyle \frac{\partial^{L}F(x,y)}{\partial p}$ defined by
  \begin{eqnarray*}\frac{\partial^{L}F(x,y)}{\partial p}=\left\{u\in Y:\lim_{\delta\rightarrow 0^{+}}\frac{1}{\delta}d\big(y+\delta u,F(x+\delta p )\big)=0\right\}
  \end{eqnarray*} is called
  lower derivative set of the set valued map $F(\cdot)$ at the point $(x,y)$ in the direction $p.$
  \end{definition}

It is obvious that for given set valued $F(\cdot):X\rightsquigarrow Y$ the inclusions
  \begin{eqnarray*}\frac{\partial^{L}F(x,y)}{\partial p}\subset\frac{\partial^{U}F(x,y)}{\partial p} ,
  \end{eqnarray*}
  \begin{eqnarray*}
  D^{L}F(x,y)|(p)\subset D^{U}F(x,y)|(p)
  \end{eqnarray*} are satisfied for every $(x,y)\in gr F(\cdot)$ and $p\in X.$

The following propositions characterize lower and upper derivative sets and differentials of the set valued maps.

\begin{proposition} \label{prop22}
Let $F(\cdot):X\rightsquigarrow Y$ be a set valued map, $gr F(\cdot)$ be a closed set, $(x,y)\in gr F(\cdot)$. Then for every $ p\in X$ the inclusions
  \begin{eqnarray*}\frac{\partial^{L}F(x,y)}{\partial p}\subset D^{L}F(x,y)|(p), \ \ \frac{\partial^{U}F(x,y)}{\partial p}\subset D^{U}F(x,y)|(p)
  \end{eqnarray*}
  are verified.

  If $F(\cdot):X\rightsquigarrow Y$ is a locally Lipschitz continuous set valued map, then for every $ p\in X$ the equalities
  \begin{eqnarray*}\frac{\partial^{L}F(x,y)}{\partial p}=D^{L}F(x,y)|(p) \ , \ \ \frac{\partial^{U}F(x,y)}{\partial p}=D^{U}F(x,y)|(p)
  \end{eqnarray*}
  hold.
\end{proposition}

Note that if $F(\cdot)$ is not a locally Lipschitz continuous set-valued map, then the equality $\displaystyle \frac{\partial^{L}F(x,y)}{\partial p}=D^{L}F(x,y)|(p)$ is not valid.

\begin{example} \label{ex3}
  Let $X=Y=\mathbb{R}$ and set-valued map $F(\cdot):\mathbb{R}\rightsquigarrow \mathbb{R}$ be defined as
  \begin{eqnarray}\label{eq21}
F(x)= \left\{
\begin{array}{lllll}
  \displaystyle x \cdot sin\frac{1}{x} & , &  x \in \mathbb{R} \setminus \left\{0\right\}\\
  \displaystyle 0 & , &  x= 0.
  \end{array}
  \right.
  \end{eqnarray}

The map $F(\cdot):\mathbb{R}\rightsquigarrow \mathbb{R}$ defined by (\ref{eq21}) is not locally Lipschitz continuous on $\mathbb{R}.$

  Since
\begin{eqnarray*}
   \frac{1}{\delta} d(0+\delta \cdot 0,F(0+\delta\cdot 1)=
   \frac{1}{\delta}d(0,F(\delta))
 =\frac{1}{\delta}\cdot \delta \left|sin\frac{1}{\delta}\right|
  =\left|sin\frac{1}{\delta}\right|
 \end{eqnarray*}
for every $\delta >0,$ then we conclude that $\displaystyle \lim_{\delta\rightarrow 0^{+}}\frac{1}{\delta}d(0+\delta \cdot 0,F(0+\delta\cdot 1)$ does not exist, and hence
  \begin{eqnarray}\label{eqad1}\displaystyle 0 \not \in \frac{\partial^{L}F(0,0)}{\partial 1}.
 \end{eqnarray}

  Now let us show that $0\in D^{L}F(0,0)|(1)$, which is equivalent to the inclusion $(1,0)\in T_{grF(\cdot)}^{L}(0,0)$.

  Choose an arbitrary sequence   $\left\{\delta_{k}\right\}_{k=1}^{\infty}$  such that  $\delta_k \rightarrow 0^{+}$ as $k\rightarrow +\infty$.
  Then for each $k$ there exists $i_k$ such that $\displaystyle \delta_{k}\in\left(\frac{1}{\pi (i_{k}+1)},\frac{1}{\pi i_{k}}\right].$ Since
  $\delta_k \rightarrow 0^{+}$ as $k\rightarrow +\infty$, then $i_k \rightarrow +\infty$ as $k\rightarrow +\infty.$ It is obvious that
 \begin{equation}\label{eq22}
   \frac{1}{\delta_{k}}<\pi (i_{k}+1)
 \end{equation}
 and
 \begin{eqnarray}\label{eq23}
   d\big( (\delta_{k},0),grF(\cdot) \big)  \leq \frac{1}{\pi i_{k}}-\frac{1}{\pi (i_{k}+1)}.
   \end{eqnarray}

   From (\ref{eq22}) and (\ref{eq23}) it follows
   \begin{eqnarray} \label{eq24} \lim_{k \rightarrow \infty}\frac{1}{\delta_{k}}d \big((0,0)+\delta_{k}(1,0),gr F(\cdot)\big) &=&
   \lim_{k \rightarrow \infty}\frac{1}{\delta_{k}}d((\delta_{k},0),gr F(\cdot)) \nonumber \\ & \leq & \lim_{k\rightarrow \infty}\pi (i_{k}+1)\left[\frac{1}{\pi i_{k}}-\frac{1}{\pi (i_{k}+1)}\right] \nonumber \\
   &=&\lim_{k\rightarrow \infty}\pi (i_{k}+1)\frac{\pi}{\pi i_{k}(i_{k}+1)} \nonumber \\
   &=&\lim_{k\rightarrow \infty}\frac{\pi}{i_{k}}  =0.
 \end{eqnarray}

 Since  $\left\{\delta_{k}\right\}_{k=1}^{\infty}$ is  arbitrarily chosen, then (\ref{eq24}) implies that
\begin{eqnarray*}
\lim_{\delta \rightarrow 0^+}\frac{1}{\delta}d \big((0,0)+\delta(1,0),gr F(\cdot)\big) =0,
\end{eqnarray*} and hence $(1,0) \in T_{grF(\cdot)}^{L}(0,0)$. Thus,

\begin{eqnarray}\label{eqad2}0 \in D^{L}F(0,0)|(1).
\end{eqnarray}

(\ref{eqad1}) and (\ref{eqad2}) yield that
\begin{eqnarray*} D^{L}F(0,0)|(1)\neq \frac{\partial^{L}F(0,0)}{\partial 1} .
\end{eqnarray*}
\end{example}

\begin{theorem} \label{teo23}
  Let $F:\mathbb{R}\rightsquigarrow Y$ be a set valued map,  $gr F(\cdot)$ be a closed set, $(x,y)\in grF(\cdot)$. Then for each
$ p \in \mathbb{R}\setminus \left\{0\right\}$ the equality
\begin{eqnarray*}
\frac{\partial^{U}F(x,y)}{\partial p}=D^{U}F(x,y)|(p)
\end{eqnarray*}
holds.
\end{theorem}
\begin{proof}
  By virtue of Proposition \ref{prop22} we have
  \begin{eqnarray}\label{eq216}
  \frac{\partial^{U}F(x,y)}{\partial p}\subset D^{U}F(x,y)|(p).
  \end{eqnarray}
  Let us prove that
  \begin{eqnarray}\label{eq217}
  D^{U}F(x,y)|(p)\subset\frac{\partial^{U}F(x,y)}{\partial p}.
  \end{eqnarray}

   Choose an arbitrary $v\in D^{U}F(x,y)|(p)$. Then  $(p,v) \in T_{grF(\cdot)}^{U}(x,y)$. According to the definition of $T_{grF(\cdot)}^{U}(x,y)$, there exist sequences $\left\{\delta_{k}\right\}_{k=1}^{\infty}$ and  $\left\{(s_{k},q_{k})\right\}_{k=1}^{\infty}$ such that $\delta_k \rightarrow 0^+,$ $(s_k,q_k)\rightarrow (0,0)$ as  $k\rightarrow +\infty$ and
    \begin{eqnarray}\label{eq218}
    (x,y)+\delta_{k}(p,v)+\delta_{k}(s_{k},q_{k}) \in gr F(\cdot)
     \end{eqnarray} for every $k=1,2,\ldots$.  Let
  \begin{eqnarray}\label{eq220}\beta_{k}=\frac{p+s_{k}}{p}\cdot \delta_{k}
   \end{eqnarray}
   Since $p\neq 0$ and $s_k \rightarrow 0$ as $k\rightarrow +\infty$, then without loss of generality it is possible to assume that $\beta_k >0$ for every $k=1,2,\ldots$ It follows from (\ref{eq220}) that $\beta_k \rightarrow 0^+$ as $k\rightarrow +\infty.$

  From (\ref{eq218}) and (\ref{eq220}) we obtain that
 \begin{eqnarray}\label{eq221}
 \left(x+\beta_{k}p,y+\frac{p}{p+s_{k}}\beta_{k}v+\frac{p }{p+s_{k}}\beta_{k} q_{k}\right)\in gr F(\cdot)
 \end{eqnarray}
 for every $k=1,2,\ldots$.

 Denote
\begin{eqnarray}\label{eq222}b_{k}=\frac{p}{p+s_{k}}v-v+\frac{p}{p+s_{k}}q_{k}.
\end{eqnarray}
 It is obvious that $b_{k} \rightarrow 0$ as $k\rightarrow +\infty$.   (\ref{eq221}) and  (\ref{eq222}) imply that \begin{eqnarray*}(x+\beta_{k}p,y+\beta_{k}v+\beta_{k}b_{k})\in grF(\cdot)
 \end{eqnarray*} and hence
 \begin{eqnarray}\label{eq223} y+\beta_{k}v+\beta_{k}b_{k}\in F(x+\beta_{k}p)
\end{eqnarray}
 for every $k=1,2,\ldots$ (\ref{eq223}) yields that
\begin{eqnarray*}
  \lim_{k\rightarrow \infty}\frac{1}{\beta_{k}}d(y+\beta_{k}v,F(x+\beta_{k}p)) &\leq & \lim_{k\rightarrow \infty}\frac{1}{\beta_{k}}[d(y+\beta_{k}v,y+\beta_{k}v+\beta_{k}b_{k})\\
  & + & d(y+\beta_{k}v+\beta_{k}b_{k},F(x+\beta_{k}p)]\\
  &=& \lim_{k\rightarrow \infty}\frac{1}{\beta_{k}}\beta_{_{k}}\|b_{k}\|=0,
\end{eqnarray*} and consequently
\begin{eqnarray}\label{eq224}\liminf_{\beta\rightarrow 0^+}\frac{1}{\beta}d(y+\beta v,F(x+\beta p))=0.
\end{eqnarray}\

We have from (\ref{eq224})
\begin{eqnarray*}v\in \frac{\partial^{U}F(x,y)}{\partial p}.
\end{eqnarray*}
Since  $v\in D^{U}F(x,y)|(p)$ is arbitrarily chosen, we obtain validity of the inclusion (\ref{eq217}). (\ref{eq216}) and (\ref{eq217}) complete the proof.
\end{proof}

\setcounter{equation}{0}

\section{Properties of the Compact Subsets of the Directional Derivative Sets}

In this section, the relations between the compact subsets of the directional derivative sets and the given set valued map are studied.

The Hausdorff deviation of the set  $E$ from the set $D$ is denoted by $h^*(E,D)$ and defined as
\begin{eqnarray*}
\displaystyle h^*(E,D)=\sup_{x\in E}\, d\left(x, D\right),
\end{eqnarray*}
where $d\left(x, D\right)$ is the distance from the point $x$ to the set $D$. If
$h^*(E,D)< r$, then  the inclusion  $E\subset D + rB$ is verified, where $B=\left\{x\in X: \left\|x\right\| \leq 1\right\}.$

\begin{theorem}\label{teo31}
 Let $F:X\rightsquigarrow Y$ be a set valued map,   $grF(\cdot)$ be a closed set, $p \in X$,  $G\subset Y$ be a compact set and  $(x,y)\in grF(\cdot) .$ Assume that the inclusion
 \begin{equation}\label{kosul}
  G\subset \frac{\partial^L F(x,y)}{\partial p}
 \end{equation}
  is satisfied. Then the equality
 \begin{equation}\label{sonuc}
 \lim_{\delta\rightarrow 0}\frac{1}{\delta}h^*\left( y+\delta G, F(x+\delta p)\right)=0
 \end{equation}
 is valid.
 \end{theorem}

\begin{proof}
 Let us assume the contrary, i.e., let the equality  $(\ref{sonuc})$ do not be satisfied. Then  there exist a sequence $\left\{\delta_i\right\}_{i=1}^{\infty}$ and $\alpha_* >0$ such that $\delta_i \rightarrow 0^+$ as $i\rightarrow +\infty$ and
 \begin{eqnarray} \label{eq313}
  \lim_{i\rightarrow \infty}\frac{1}{\delta_i}h^*\left(y+\delta_i G, F(x+\delta_i p)\right)=\alpha_* .
 \end{eqnarray}

   Let $\mu_* <\alpha_*$ be an arbitrary number. It follows from  $(\ref{eq313})$ that
  there exists  $N_1>0$ such that
  \begin{eqnarray*} h^*\left(y+\delta_i G, F(x+\delta_i p)\right)> \frac{\mu_*}{2}\delta_i
   \end{eqnarray*}
 for every  $i>N_1$.  Then for each $i>N_1$ there exists a $g_i\in G$ such that
  \begin{eqnarray*}
  y+\delta_i g_i \notin F(x+\delta_i p) + \frac{\mu_*}{2}\delta_i B,
   \end{eqnarray*}
 and hence
 \begin{eqnarray}\label{eq314}
   d\left( y+\delta_i g_i, F(x+\delta_i p)\right)>\dfrac{\mu_*}{4}\delta_i .
 \end{eqnarray}

Since $G\subset Y$ is a compact set, $g_i\in G$ for every $i=1,2,\ldots,$ then without loss of generality we may assume that $g_i\rightarrow g_*$ as $i\rightarrow \infty$ and $g_*\in G$.

According to (\ref{kosul}) we have  $\displaystyle g_* \in \frac{\partial^L F(x,y)}{\partial p},$ and therefore
\begin{eqnarray}\label{eq315}\lim_{\delta\rightarrow 0}\dfrac{1}{\delta}d\left(y+\delta g_*, F(x+\delta p)\right)=0.
\end{eqnarray}

Since $g_i\rightarrow g_*$ as $i\rightarrow \infty$, then it follows from (\ref{eq315}) that
\begin{eqnarray}\label{eq316}
 \displaystyle \lim_{i\rightarrow \infty}\frac{1}{\delta_i}d\left( y+\delta_i g_i, F(x+\delta_i p) \right)
  &\leq& \displaystyle \lim_{i\rightarrow \infty}\frac{1}{\delta_i}\big[ d\left(y +\delta_i g_i, y+\delta_i g_* \right) \nonumber \\ \qquad &+&
        d\left( y+\delta_i g_*, F(x+\delta_i p) \right)\big] \nonumber \\
  &=&  \lim_{i\rightarrow \infty}\left\|g_i-g_*\right\|= 0.
\end{eqnarray}
  $(\ref{eq314})$ and $(\ref{eq316})$ contradict. Proof is completed.
\end{proof}

\begin{corollary}\label{cor1}
  Suppose that the conditions of Theorem \ref{teo31} are satisfied. Then there exist $\delta_* >0 $  and a function $r(\cdot):[0,\delta_*]\rightarrow [0,\infty)$ such that  $r(\delta )\rightarrow 0^+ $ as $\delta \rightarrow 0^+$ and
\begin{eqnarray*} y+\delta G\subset F(x+\delta p) + \delta r(\delta)B
\end{eqnarray*}
for every $\delta \in [0,\delta_*].$
\end{corollary}

Theorem  \ref{teo31} is not true if in (\ref{kosul}) the lower directional derivative set will be replaced by the upper directional derivative set.
\begin{example} \label{ex4}
 Let the set-valued map $F(\cdot):\mathbb{R}\rightsquigarrow \mathbb{R}$ be defined as in Example \ref{ex3} by (\ref{eq21}), $p=1,$ $(x,y)=(0,0) \in gr F(\cdot)$.
   One can show that $\displaystyle \frac{\partial^U F(0,0)}{\partial 1}=[-1,1].$  Let $G=\left[-1,1\right].$ Then  $\displaystyle G\subset \frac{\partial^U F(0,0)}{\partial 1}$, but
 \begin{eqnarray*}
  \displaystyle
  \liminf_{\delta\rightarrow 0^+}\frac{1}{\delta}h^*\left( 0+\delta G, F(0+\delta \cdot 1)\right)&=&
  \liminf_{\delta\rightarrow 0^+}\frac{1}{\delta}h^*\left( 0+\delta [-1,1] , \delta\sin \frac{1}{\delta}\right)\nonumber \\
  &=&\displaystyle  \liminf_{\delta\rightarrow 0^+}\frac{1}{\delta}h^*\left( \left[-\delta,\delta\right] , \delta\sin \frac{1}{\delta}\right)\nonumber \\
  &=& \displaystyle\liminf_{\delta\rightarrow 0^+}\frac{1}{\delta} \sup_{t\in \left[-\delta,\delta \right]}\, \left|t-\sin \frac{1}{\delta}\right| \nonumber \\
  &\geq& \frac{1}{2} \ ,
\end{eqnarray*}
and  equality (\ref{sonuc}) is not satisfied.
\end{example}

Note, that in Theorem \ref{teo31} the lower directional derivative set $\displaystyle \frac{\partial^L F(x,y)}{\partial p}$ cannot be replaced by the set $\displaystyle D^L F(x,y)|(p)$.

\begin{example} \label{ex5}
 Let the set-valued map $F(\cdot):\mathbb{R}\rightsquigarrow \mathbb{R}$ be defined as in Example \ref{ex3} by (\ref{eq21}), $p=1,$ $(x,y)=(0,0) \in gr F(\cdot)$.
  According to  Example \ref{ex3}, $\displaystyle 0 \in D^L F(0,0)|(1).$  Let $G=\left\{0\right\}.$ Then  $\displaystyle G\subset D^U F(0,0)|(1)$.

Since
\begin{eqnarray*}
  \displaystyle \frac{1}{\delta} h^*\left( 0+\delta G, F(0+\delta \cdot 1)\right)
  =\displaystyle \frac{1}{\delta} h^*\left(0 , \delta\sin \frac{1}{\delta}\right)
  = \displaystyle \frac{1}{\delta} \left|\delta\sin \frac{1}{\delta} \right| = \left|\sin \frac{1}{\delta} \right|
\end{eqnarray*} for every $\delta >0,$  the limit  $\displaystyle\lim_{\delta\rightarrow 0+}\frac{1}{\delta}h^*\left( 0+\delta G, F(0+\delta \cdot 1)\right)$ does not exist which verifies that equality (\ref{sonuc}) is  not satisfied.
\end{example}
\newpage
\begin{flushright}
Anadolu University, Mathematics Department, \\26470 Eskisehir, Turkey.\\
sizgi@anadolu.edu.tr\\
nsahin@anadolu.edu.tr\\
ahuseyin@anadolu.edu.tr\\
nhuseyin@anadolu.edu.tr
\end{flushright}

\end{document}